\documentclass[11pt]{amsart}

\usepackage{amsmath,amssymb,amsthm}
\usepackage[margin=1.1in]{geometry}
\usepackage[colorlinks=true,linkcolor=blue,citecolor=blue,urlcolor=blue]{hyperref}

\newtheorem{theorem}{Theorem}[section]
\newtheorem{lemma}[theorem]{Lemma}
\newtheorem{proposition}[theorem]{Proposition}
\newtheorem{corollary}[theorem]{Corollary}
\theoremstyle{definition}

\theoremstyle{remark}
\newtheorem{remark}[theorem]{Remark}
\newtheorem{question}[theorem]{Question}

\newcommand{\K}{\mathbb{K}}
\newcommand{\R}{\mathbb{R}}
\newcommand{\C}{\mathbb{C}}
\newcommand{\Z}{\mathbb{Z}}
\newcommand{\N}{\mathbb{N}}
\newcommand{\Q}{\mathbb{Q}}
\newcommand{\T}{\mathcal{T}}
\newcommand{\D}{\mathcal{D}}
\newcommand{\U}{\mathcal{U}}
\newcommand{\diag}{\operatorname{diag}}
\newcommand{\cl}[1]{\overline{#1}}
\newcommand{\gen}[1]{\langle #1\rangle}

\makeatletter
\@namedef{subjclassname@2020}{\textup{2020} Mathematics Subject Classification}
\makeatother

\title{Dense semigroups of triangular matrices}
\author{Mohammad Javaheri}
\address{Department of Mathematics and Statistics, Siena University\\
515 Loudon Road, Loudonville, NY 12211, USA}
\email{mjavaheri@siena.edu}
\date{Revised, November 2026}


\begin{document}
\begin{abstract}
Let $\K=\R$ or $\C$, and let $\T_n(\K)$ be the semigroup of $n\times n$
lower triangular matrices with entries in $\K$. We show that $\T_n(\K)$ has
dense subsemigroups generated by $n+1$ matrices, and that $n+1$ is the least
possible number of generators; the same holds for dense subgroups of the group
of invertible lower triangular matrices. The generators can be written down
explicitly: $n$ of them are diagonal and one is a diagonal matrix plus a
nilpotent perturbation. More precisely, let $D_0,\dots,D_m$ be diagonal
matrices generating a dense subsemigroup of the diagonal matrices, where the
diagonal entries of $D_0$ have strictly monotone moduli, and let $A$ be a lower
triangular matrix with the same diagonal as $D_0$. We prove that the
semigroup generated by $A,D_1,\dots,D_m$ is dense in $\T_n(\K)$ if and only if
every subdiagonal entry of $A$ is nonzero.
\end{abstract}
\maketitle

\section{Introduction}

Dense subgroups of Lie groups have been studied by many mathematicians, dating
back to Auerbach, who proved that every compact semisimple Lie group has a
$2$-generator dense subgroup~\cite{Au}. Kuranishi~\cite{Ku} proved that if $G$
is a semisimple Lie group and $a,b$ are two elements near the identity, then
the subgroup generated by $a$ and $b$ is dense in $G$ if and only if $\log a$
and $\log b$ generate the Lie algebra of $G$. More recently, Abels and
Vinberg~\cite{AV} showed that every connected semisimple Lie group with finite
center has $2$-generator dense subsemigroups, and Breuillard and
Gelander~\cite{BG} proved that every dense subgroup of a connected semisimple
real Lie group contains a $2$-generator dense subgroup.

Dense subgroups of linear groups have also been studied in particular cases.
Wang~\cite{Wa} has shown that every dense subgroup of the group of
orientation-preserving M\"obius transformations of the $n$-dimensional unit
sphere, $n\ge 2$, contains a dense subgroup generated by at most $n$
elements. A similar statement was proved by Cao~\cite{Ca} for $U(n,1)$.
Finally, $2$-generator dense subsemigroups of the semigroup of all $n\times n$
matrices can be constructed in both the real and the complex
case~\cite{Ja2}.

In connection with hypercyclicity, the following results have inspired this
paper. Feldman proved that there is a hypercyclic $(n+1)$-tuple of commuting
diagonal $n\times n$ matrices, and hence a dense subsemigroup of the diagonal
matrices generated by $n+1$ elements~\cite{Fe}. Abels and Manoussos~\cite{AM}
showed that the minimal number of generators of an abelian semigroup of
triangular non-diagonalizable matrices with a dense or somewhere dense orbit
is $n+1$ in the real case and $n+2$ in the complex case.

Throughout, $\K=\R$ or $\C$, $\T_n=\T_n(\K)$ is the set of $n\times n$ lower
triangular matrices, $\D_n=\D_n(\K)$ is the set of $n\times n$ diagonal
matrices, and $\diag(a_1,\dots,a_n)$ is the diagonal matrix with diagonal
entries $a_1,\dots,a_n$. For a subset $\mathcal{G}$ of a semigroup we write
$\gen{\mathcal{G}}$ for the subsemigroup generated by $\mathcal{G}$ (all finite
nonempty products of elements of $\mathcal{G}$). All topologies are the
Euclidean ones. The following theorem is the main result of this article.

\begin{theorem}\label{thm:main}
Let $n\ge 1$ and $\K=\R$ or $\C$.
\begin{enumerate}
\item[(a)] $\T_n(\K)$ contains a dense subsemigroup generated by $n+1$
matrices, and no dense subsemigroup of $\T_n(\K)$ is generated by $n$ or fewer
matrices.
\item[(b)] Explicitly, let $p_1>p_2>\dots>p_n$ and $q_1,\dots,q_n$ be $2n$
distinct primes, let $N$ be the matrix with $N_{i+1,i}=1$ for $1\le i\le n-1$
and all other entries $0$, and let $E_{jj}$ be the matrix unit. Put
\[
A=\diag\bigl(e^{-\sqrt{p_1}},\dots,e^{-\sqrt{p_n}}\bigr)+N,\qquad
D_j=I-(1+e)E_{jj}\quad(1\le j\le n)
\]
in the real case, and
\[
A=\diag\bigl(e^{-\sqrt{p_1}+2\pi i\sqrt{q_1}},\dots,e^{-\sqrt{p_n}+2\pi i\sqrt{q_n}}\bigr)+N,
\qquad D_j=I+(e-1)E_{jj}\quad(1\le j\le n)
\]
in the complex case. Then $\gen{A,D_1,\dots,D_n}$ is dense in $\T_n(\K)$.
\end{enumerate}
\end{theorem}

Theorem~\ref{thm:main} is deduced from the following characterization, which
is the technical heart of the paper.

\begin{theorem}\label{thm:char}
Let $D_0,D_1,\dots,D_m$ be diagonal $n\times n$ matrices with entries in $\K$
such that $\gen{D_0,\dots,D_m}$ is dense in $\D_n(\K)$. Suppose that
$D_0=\diag(a_1,\dots,a_n)$, where either
\begin{equation}\label{eq:incr}
0<|a_1|<|a_2|<\dots<|a_n|\qquad\text{or}\qquad 0<|a_n|<|a_{n-1}|<\dots<|a_1|.
\end{equation}
Let $A\in\T_n(\K)$ have diagonal $D_0$, i.e.\ $A_{ii}=a_i$ for all $i$. Then
$\gen{A,D_1,\dots,D_m}$ is dense in $\T_n(\K)$ if and only if
\begin{equation}\label{eq:subdiag}
A_{i+1,i}\neq 0\qquad\text{for all } 1\le i\le n-1 .
\end{equation}
\end{theorem}

Condition~\eqref{eq:subdiag} is necessary for any semigroup of this form, so
Theorem~\ref{thm:char} is optimal as far as the off-diagonal part of $A$ is
concerned; see also Remark~\ref{rem:nec}. In particular, it is not necessary
for all entries of $A$ below the diagonal to be nonzero.

The paper is organized as follows. In Section~\ref{sec:diag} we prove the
lower bound $n+1$ and construct explicit dense $(n+1)$-generator semigroups of
diagonal matrices with the monotonicity property~\eqref{eq:incr}. In
Section~\ref{sec:lemmas} we collect the matrix lemmas, and in
Section~\ref{sec:proofs} we prove Theorems~\ref{thm:char}
and~\ref{thm:main}. Section~\ref{sec:cor} contains corollaries and open
questions.

\section{Dense semigroups of diagonal matrices}\label{sec:diag}

\subsection{Removing singular generators}

\begin{lemma}\label{lem:singular}
Let $\mathcal{G}\subseteq\D_n$ be finite with $\gen{\mathcal{G}}$ dense in
$\D_n$, and let $\mathcal{G}^\ast$ be the set of invertible elements of
$\mathcal{G}$. Then $\gen{\mathcal{G}^\ast}$ is dense in $\D_n$.
\end{lemma}

\begin{proof}
Let $Z=\{X\in\D_n:\det X=0\}$, a closed subset with empty interior. Every
element of $\gen{\mathcal{G}}$ that involves a singular generator lies in $Z$.
Hence $\D_n=\cl{\gen{\mathcal{G}}}\subseteq \cl{\gen{\mathcal{G}^\ast}}\cup Z$,
so $\cl{\gen{\mathcal{G}^\ast}}$ contains the dense set $\D_n\setminus Z$ and
therefore equals $\D_n$.
\end{proof}

\subsection{The lower bound}

\begin{proposition}\label{prop:lower}
\begin{enumerate}
\item[(a)] A dense subsemigroup of $\D_n(\K)$ cannot be generated by fewer than
$n+1$ elements.
\item[(b)] A dense subsemigroup of $\T_n(\K)$ cannot be generated by fewer than
$n+1$ elements.
\item[(c)] A dense subgroup of the group $\D_n^\ast(\K)$ of invertible diagonal
matrices, or of the group $\T_n^\ast(\K)$ of invertible lower triangular
matrices, cannot be generated by fewer than $n+1$ elements.
\end{enumerate}
\end{proposition}

\begin{proof}
Let $L:\D_n^\ast\to(\R^n,+)$ be the continuous surjective homomorphism
$L(\diag(x_1,\dots,x_n))=(\log|x_1|,\dots,\log|x_n|)$, and let
$\pi:\T_n\to\D_n$ be the map sending a matrix to its diagonal part. Since the
diagonal of a product of lower triangular matrices is the product of the
diagonals, $\pi$ is a continuous surjective semigroup homomorphism, and it
maps $\T_n^\ast$ onto $\D_n^\ast$.

(a) Suppose $\gen{\mathcal{G}}$ is dense in $\D_n$ with $|\mathcal{G}|=k$. By
Lemma~\ref{lem:singular} the invertible generators $\mathcal{G}^\ast$ also
generate a dense subsemigroup, which is then dense in the open set
$\D_n^\ast$. Hence the additive semigroup $\Sigma$ generated by the vectors
$w\in L(\mathcal{G}^\ast)$ is dense in $\R^n$. If these vectors do not span
$\R^n$, then $\Sigma$ lies in a proper subspace, which is impossible. If they
span $\R^n$ and there are exactly $n$ of them, they form a basis and $\Sigma$
lies in the closed cone $\{\sum t_iw_i:t_i\ge0\}$, whose complement contains
the nonempty open set $\{\sum t_iw_i:t_i<0\}$; again impossible. Hence
$k\ge|\mathcal{G}^\ast|\ge n+1$.

(b) If $\gen{\mathcal{G}}$ is dense in $\T_n$, then
$\pi(\gen{\mathcal{G}})=\gen{\pi(\mathcal{G})}$ is dense in $\D_n$, so
$|\mathcal{G}|\ge n+1$ by (a).

(c) Applying $L$ (and first $\pi$ in the triangular case), a dense subgroup
generated by $k$ elements yields a dense subgroup of $\R^n$ generated by $k$
vectors. If $k\le n$, either these vectors span a proper subspace, or $k=n$
and they form a basis, in which case they generate a lattice, which is
discrete. Hence $k\ge n+1$.
\end{proof}

\subsection{A density criterion}

The following lemma is a version of a classical argument (compare
\cite{Fe,AM}); we include a proof for completeness.

\begin{lemma}\label{lem:positive}
Let $G=\R^d\times C$, where $C$ is a compact metrizable abelian group, written
additively, and let $\pi_1:G\to\R^d$ be the projection. Let
$g_1,\dots,g_N\in G$ be such that
\begin{enumerate}
\item[(i)] the subgroup $\Gamma$ generated by $g_1,\dots,g_N$ is dense in $G$, and
\item[(ii)] there are real numbers $c_1,\dots,c_N>0$ with
$\sum_{i=1}^N c_i\,\pi_1(g_i)=0$.
\end{enumerate}
Then the subsemigroup $S$ generated by $g_1,\dots,g_N$ is dense in $G$.
\end{lemma}

\begin{proof}
Fix a translation-invariant metric $\rho$ on $C$ compatible with its topology
and let 
$$\delta((x,\kappa),(x',\kappa'))=\|x-x'\|+\rho(\kappa,\kappa'),$$ 
a translation-invariant metric on $G$. Let $\Phi:\Z^N\to G$,
$\Phi(m)=\sum m_ig_i$, so that $\Gamma=\Phi(\Z^N)$ and
$S=\Phi(\N^N\setminus\{0\})$, where $\N=\{0,1,2,\dots\}$.

\emph{Claim.} For every $\varepsilon>0$ and $M\in\N$ there is $q\in\Z^N$ with
$q_i\ge M$ for all $i$ and $\delta(\Phi(q),0)<\varepsilon$.

For $t\in\N$ let $q(t)=(\lfloor tc_1\rfloor,\dots,\lfloor tc_N\rfloor)$ and
$y_t=\Phi(q(t))$. By (ii),
$\pi_1(y_t)=\sum_i(\lfloor tc_i\rfloor-tc_i)\pi_1(g_i)$ has norm at most
$\sum_i\|\pi_1(g_i)\|$, so $(y_t)$ lies in a compact subset of $G$ and has a
convergent subsequence $(y_{t_k})$. Choose $k<l$ with
$(t_l-t_k)\min_ic_i\ge M+1$ and $\delta(y_{t_l},y_{t_k})<\varepsilon$, and let
$q=q(t_l)-q(t_k)$. Then $q_i\ge(t_l-t_k)c_i-1\ge M$, and by invariance
$\delta(\Phi(q),0)=\delta(y_{t_l},y_{t_k})<\varepsilon$.

Now let $m\in\Z^N$ and $\varepsilon>0$. Take $q$ as in the claim with
$M=1+\max_i|m_i|$. Then $m+q$ has all coordinates $\ge1$, so
$\Phi(m+q)\in S$, and $\delta(\Phi(m+q),\Phi(m))=\delta(\Phi(q),0)<\varepsilon$.
Thus $\Gamma\subseteq\cl S$, and by (i), $\cl S=G$.
\end{proof}

\subsection{An explicit construction}

Recall Kronecker's theorem~\cite[Thm.~442]{HW}: if
$1,\theta_1,\dots,\theta_k$ are linearly independent over $\Q$, then
$\{(m\theta_1,\dots,m\theta_k)\bmod \Z^k:m\in\Z\}$ is dense in $\R^k/\Z^k$.
We also use Besicovitch's theorem~\cite{Be}: the square roots of distinct
primes, together with $1$, are linearly independent over $\Q$.

\begin{proposition}\label{prop:diag}
Let $c_1,\dots,c_n>0$.
\begin{enumerate}
\item[(a)] \emph{(Real case.)} If $1,c_1,\dots,c_n$ are linearly independent
over $\Q$, then the matrices
\[
D_0=\diag(e^{-c_1},\dots,e^{-c_n}),\qquad D_j=I-(1+e)E_{jj}\ \ (1\le j\le n)
\]
generate a dense subsemigroup of $\D_n(\R)$. (Thus $D_j$ is diagonal with
$-e$ in position $j$ and $1$ elsewhere.)
\item[(b)] \emph{(Complex case.)} If $\beta_1,\dots,\beta_n\in\R$ and
$1,c_1,\dots,c_n,\beta_1,\dots,\beta_n$ are linearly independent over $\Q$,
then the matrices
\[
D_0=\diag(e^{-c_1+2\pi i\beta_1},\dots,e^{-c_n+2\pi i\beta_n}),\qquad
D_j=I+(e-1)E_{jj}\ \ (1\le j\le n)
\]
generate a dense subsemigroup of $\D_n(\C)$.
\end{enumerate}
In both cases, if $c_1>c_2>\dots>c_n$, then $D_0$ satisfies
$0<|a_1|<\dots<|a_n|$.
\end{proposition}

\begin{proof}
Since $\D_n^\ast$ is dense in $\D_n$, it suffices to prove density in
$\D_n^\ast$. Let $c=(c_1,\dots,c_n)$ and let $e_1,\dots,e_n$ be the standard
basis of $\R^n$.

(a) The map $E:\R^n\times(\Z/2)^n\to\D_n^\ast(\R)$,
$E(x,\epsilon)=\diag((-1)^{\epsilon_1}e^{x_1},\dots,(-1)^{\epsilon_n}e^{x_n})$,
is an isomorphism of topological groups, and $D_0=E(g_0)$, $D_j=E(g_j)$ with
$g_0=(-c,0)$ and $g_j=(e_j,e_j)$. Condition (ii) of Lemma~\ref{lem:positive}
holds with coefficients $(1,c_1,\dots,c_n)$. For (i), let
$(x,\epsilon)\in\R^n\times\{0,1\}^n$ and $\eta>0$. Since
$1,c_1/2,\dots,c_n/2$ are linearly independent over $\Q$, Kronecker's theorem
gives $m_0\in\Z$ and $p\in\Z^n$ with
$\|p-m_0c/2-(x-\epsilon)/2\|<\eta/2$. Put $m_j=2p_j+\epsilon_j$. Then
$m_0g_0+\sum_jm_jg_j=(m-m_0c,\epsilon)$ with $m=(m_1,\dots,m_n)$, and
$\|m-m_0c-x\|=2\|p-m_0c/2-(x-\epsilon)/2\|<\eta$. Hence the group generated by
the $g_j$ is dense, and Lemma~\ref{lem:positive} applies.

(b) The map $E:\R^n\times(\R/\Z)^n\to\D_n^\ast(\C)$,
$E(x,\theta)=\diag(e^{x_1+2\pi i\theta_1},\dots,e^{x_n+2\pi i\theta_n})$, is an
isomorphism of topological groups, and $D_0=E(g_0)$, $D_j=E(g_j)$ with
$g_0=(-c,\beta)$ and $g_j=(e_j,0)$. Again (ii) holds with coefficients
$(1,c_1,\dots,c_n)$. For (i), given $(x,\theta)$ and $\eta>0$, Kronecker's
theorem applied to $(c,\beta)\in\R^{2n}$ gives $m_0\in\Z$ such that
$m_0c+x$ is within $\eta$ of some $m\in\Z^n$ and $m_0\beta$ is within $\eta$
of $\theta$ modulo $\Z^n$. Then $m_0g_0+\sum_jm_jg_j=(m-m_0c,m_0\beta)$ is
within $2\eta$ of $(x,\theta)$. Lemma~\ref{lem:positive} applies.

The last assertion is clear since $|a_j|=e^{-c_j}$.
\end{proof}

\begin{remark}
Feldman~\cite[Theorem~3.4]{Fe} gives a different construction of dense
$(n+1)$-generator semigroups of complex diagonal matrices. The point of
Proposition~\ref{prop:diag} is that it treats the real and complex cases
directly and produces a generator $D_0$ whose diagonal entries have strictly
monotone moduli, as required by Theorem~\ref{thm:char}.
\end{remark}

\section{Matrix lemmas}\label{sec:lemmas}

\begin{lemma}\label{lem:growth}
Let $A\in\T_n$ have pairwise distinct diagonal entries $a_1,\dots,a_n$. Then
there is $\lambda>0$, depending only on $A$, such that
\begin{equation}\label{eq:growth}
\bigl|(A^k)_{ij}\bigr|\le\lambda\max_{j\le l\le i}|a_l|^k\qquad
\text{for all } 1\le j\le i\le n \text{ and } k\ge0 .
\end{equation}
In particular, if $|a_1|<\dots<|a_n|$, then $|(A^k)_{ij}|\le\lambda|a_i|^k$
\textup{(}cf.~\cite{Ja1}\textup{)}.
\end{lemma}

\begin{proof}
We first show that $AP=PD$ for some lower unitriangular $P$, where
$D=\diag(a_1,\dots,a_n)$. For the $l$-th column $u$ of $P$, put $u_m=0$ for
$m<l$, $u_l=1$, and recursively, for $m>l$,
\[
u_m=\frac{1}{a_l-a_m}\sum_{l\le q<m}A_{mq}u_q ,
\]
which is possible since $a_m\neq a_l$. Then $(Au)_m=a_lu_m$ for all $m$. Hence
$A^k=PD^kP^{-1}$, where $P^{-1}$ is also lower unitriangular, and
\[
(A^k)_{ij}=\sum_{l=1}^nP_{il}\,a_l^k\,(P^{-1})_{lj}
=\sum_{j\le l\le i}P_{il}\,a_l^k\,(P^{-1})_{lj}.
\]
Thus \eqref{eq:growth} holds with
$\lambda=\max_{i,j}\sum_{l}|P_{il}||(P^{-1})_{lj}|$.
\end{proof}

Let $\Delta=\{(i,j):1\le j<i\le n\}$ be the set of strictly subdiagonal
positions. We order $\Delta$ totally by
\[
(i,j)\preceq(r,s)\iff i-j<r-s\ \text{ or }\ (i-j=r-s\ \text{and}\ i\le r).
\]
Thus $\Delta$ is ordered diagonal by diagonal, starting with
$(2,1)\prec(3,2)\prec\dots\prec(n,n-1)\prec(3,1)\prec\cdots$ and ending with
$(n,1)$. For $(r,s)\in\Delta$ we write $(r,s)^+$ for its successor, with the
convention that $(n,1)^+=\infty$, and we set
\[
\U_{(r,s)}=\{X\in\T_n:\ X_{ij}=0\ \text{for all }(i,j)\in\Delta\text{ with }(i,j)\prec(r,s)\},
\qquad \U_\infty=\D_n .
\]
Then $\U_{(2,1)}=\T_n$, each $\U_{(r,s)}$ is a closed linear subspace
containing $\D_n$, and $\U_{(r,s)^+}\subseteq\U_{(r,s)}$.

\begin{lemma}\label{lem:prod}
Let $(r,s)\in\Delta$. Then $\U_{(r,s)}$ is closed under multiplication, and
for $X,Y\in\U_{(r,s)}$,
\begin{equation}\label{eq:prod}
(XY)_{rs}=X_{rs}Y_{ss}+X_{rr}Y_{rs}.
\end{equation}
Consequently, if $X\in\U_{(r,s)}$, then for all $k\ge1$,
\begin{equation}\label{eq:power}
(X^k)_{rs}=X_{rs}\sum_{t=0}^{k-1}X_{rr}^{\,t}X_{ss}^{\,k-1-t},
\quad\text{which equals}\quad
X_{rs}\,\frac{X_{rr}^k-X_{ss}^k}{X_{rr}-X_{ss}}\ \text{ if }X_{rr}\neq X_{ss}.
\end{equation}
\end{lemma}

\begin{proof}
Let $X,Y\in\U_{(r,s)}$ and $(i,j)\in\Delta$ with $(i,j)\preceq(r,s)$. Then
$(XY)_{ij}=\sum_{j\le l\le i}X_{il}Y_{lj}$. For $j<l<i$ we have
$(i,l)\in\Delta$ and $i-l<i-j\le r-s$, so $(i,l)\prec(r,s)$ and $X_{il}=0$.
Hence $(XY)_{ij}=X_{ij}Y_{jj}+X_{ii}Y_{ij}$. If $(i,j)\prec(r,s)$ this
vanishes, which proves closure; for $(i,j)=(r,s)$ it is~\eqref{eq:prod}.
Formula~\eqref{eq:power} follows from~\eqref{eq:prod} by induction on $k$,
using $(X^{k+1})_{rs}=(X^k)_{rs}X_{ss}+X_{rr}^kX_{rs}$.
\end{proof}

\begin{lemma}\label{lem:diagclosure}
Let $D_0,D_1,\dots,D_m\in\D_n$ be such that $S=\gen{D_0,\dots,D_m}$ is dense
in $\D_n$, where $D_0=\diag(a_1,\dots,a_n)$ satisfies
$0<|a_1|<\dots<|a_n|$. Let $A\in\T_n$ with diagonal $D_0$ and let
$\Lambda=\gen{A,D_1,\dots,D_m}$. Then $\D_n\subseteq\cl\Lambda$.
\end{lemma}

\begin{proof}
Note that $\cl\Lambda$ is a closed subsemigroup of $\T_n$. We prove by
induction along $(\Delta,\preceq)$:

\smallskip\noindent
($\ast$) \emph{For every $(r,s)\in\Delta\cup\{\infty\}$, the set
$\cl\Lambda\cap\U_{(r,s)}$ contains a matrix with diagonal $D_0$.}

\smallskip
For $(r,s)=(2,1)$ the matrix $A$ itself works. Assume ($\ast$) holds for some
$(r,s)\in\Delta$, witnessed by $A'\in\cl\Lambda\cap\U_{(r,s)}$ with diagonal
$D_0$. Since $D_1,\dots,D_m\in\Lambda$, the semigroup
$\Lambda'=\gen{A',D_1,\dots,D_m}$ is contained in $\cl\Lambda$.

\emph{Step 1: approximations with large exponent of $D_0$.} Since $D_0$ is
invertible, multiplication by $D_0^k$ is a homeomorphism of $\D_n$; hence
$D_0^kS$ is dense in $\D_n$ for every $k$. Since $\D_n$ is commutative,
every element of $D_0^kS$ has the form $E\,D_0^{\,d}$ with
$E=D_1^{d_1}\cdots D_m^{d_m}$, $d_\alpha\ge0$, and $d\ge k$. Thus, for every
$B\in\D_n$ there are $E_k=D_1^{d_1^{(k)}}\cdots D_m^{d_m^{(k)}}$ and integers
$d^{(k)}\ge k$ such that
\begin{equation}\label{eq:approx}
s_k:=E_k\,D_0^{\,d^{(k)}}\longrightarrow B\qquad(k\to\infty).
\end{equation}

\emph{Step 2: the elements $\eta(B)$.} Put
$M_k=E_k\,(A')^{d^{(k)}}\in\Lambda'$. By Lemma~\ref{lem:prod},
$M_k\in\U_{(r,s)}$, and its diagonal is $s_k$. By Lemma~\ref{lem:growth},
\[
|(M_k)_{ij}|=|(E_k)_{ii}|\,|((A')^{d^{(k)}})_{ij}|
\le\lambda\,|(E_k)_{ii}|\,|a_i|^{d^{(k)}}=\lambda\,|(s_k)_{ii}|,
\]
which is bounded in $k$. By~\eqref{eq:power},
\[
(M_k)_{rs}=(E_k)_{rr}\,A'_{rs}\,\frac{a_r^{d^{(k)}}-a_s^{d^{(k)}}}{a_r-a_s}
=(s_k)_{rr}\,\frac{1-(a_s/a_r)^{d^{(k)}}}{a_r-a_s}\,A'_{rs}
\longrightarrow\frac{B_{rr}A'_{rs}}{a_r-a_s},
\]
because $s<r$ gives $|a_s/a_r|<1$ and $d^{(k)}\to\infty$. Passing to a
convergent subsequence and using that $\cl\Lambda$ and $\U_{(r,s)}$ are
closed, we obtain $\eta(B)\in\cl\Lambda\cap\U_{(r,s)}$ with diagonal $B$ and
\[
\eta(B)_{rs}=\frac{B_{rr}A'_{rs}}{a_r-a_s}.
\]

\emph{Step 3: killing the $(r,s)$ entry.} Let $B_2$ be the diagonal matrix with
$(B_2)_{rr}=-1$ and $(B_2)_{ii}=1$ for $i\neq r$, and let $B_1=B_2D_0$. Put
$M_1=\eta(B_1)$, $M_2=\eta(B_2)$ and $A''=M_1M_2\in\cl\Lambda$. By
Lemma~\ref{lem:prod}, $A''\in\U_{(r,s)}$, its diagonal is $B_1B_2=D_0$, and
by~\eqref{eq:prod},
\[
A''_{rs}=(M_1)_{rs}(B_2)_{ss}+(B_1)_{rr}(M_2)_{rs}
=\frac{-a_rA'_{rs}}{a_r-a_s}\cdot1+(-a_r)\cdot\frac{-A'_{rs}}{a_r-a_s}=0 .
\]
Hence $A''\in\U_{(r,s)^+}$, which proves ($\ast$) for $(r,s)^+$.

By induction, ($\ast$) holds for $\infty$, i.e.\ $D_0\in\cl\Lambda$. Then
$S=\gen{D_0,\dots,D_m}\subseteq\cl\Lambda$, and $\D_n=\cl S\subseteq\cl\Lambda$.
\end{proof}

\section{Proofs of the main results}\label{sec:proofs}

\begin{proof}[Proof of Theorem~\ref{thm:char}]
\emph{Necessity.} Let $1\le i\le n-1$ and suppose $A_{i+1,i}=0$. For
$X,Y\in\T_n$ we have
$(XY)_{i+1,i}=X_{i+1,i}Y_{ii}+X_{i+1,i+1}Y_{i+1,i}$. Since every generator of
$\Lambda=\gen{A,D_1,\dots,D_m}$ has a zero $(i+1,i)$ entry, so does every
element of $\Lambda$, and $\Lambda$ is contained in the proper closed subspace
$\{X\in\T_n:X_{i+1,i}=0\}$. Thus~\eqref{eq:subdiag} is necessary.

\emph{Reduction to increasing moduli.} Let $J$ be the $n\times n$ permutation
matrix with $J_{i,n+1-i}=1$, and let $\tau(X)=JX^{\mathsf T}J$, so that
$\tau(X)_{ij}=X_{n+1-j,\,n+1-i}$. Then $\tau$ is a linear homeomorphism of
$\T_n$ onto itself with $\tau(XY)=\tau(Y)\tau(X)$, so $\tau$ maps the semigroup
generated by a set $\mathcal{G}$ onto the semigroup generated by
$\tau(\mathcal{G})$. Moreover $\tau$ maps $\D_n$ onto itself, reversing the
order of the diagonal entries, and $\tau(A)_{i+1,i}=A_{n+1-i,\,n-i}$. Hence all
hypotheses are preserved by $\tau$, with the second alternative
in~\eqref{eq:incr} turned into the first. We may therefore assume
$0<|a_1|<\dots<|a_n|$.

\emph{Reduction to invertible generators.} By Lemma~\ref{lem:singular}, the
invertible ones among $D_0,\dots,D_m$ (which include $D_0$) still generate a
dense subsemigroup of $\D_n$. Discarding the others only shrinks $\Lambda$, so
we may assume that all $D_\alpha$ are invertible. Then $A$ is invertible as
well, and so is every element of $\Lambda$.

\emph{Sufficiency, by induction on $n$.} For $n=1$, $\T_1=\D_1$, $A=D_0$, and
there is nothing to prove. Let $n\ge2$ and assume the result for $n-1$. Write
$\Lambda=\gen{A,D_1,\dots,D_m}$. By Lemma~\ref{lem:diagclosure},
$\D_n\subseteq\cl\Lambda$.

For $X\in\T_n$ let $\sigma(X)\in\T_{n-1}$ be its upper left
$(n-1)\times(n-1)$ block. Since the upper right block of a lower triangular
matrix vanishes, $\sigma$ is a continuous surjective semigroup homomorphism.
The matrices $\sigma(D_0),\dots,\sigma(D_m)$ generate a dense subsemigroup of
$\D_{n-1}$, $\sigma(D_0)=\diag(a_1,\dots,a_{n-1})$ satisfies~\eqref{eq:incr},
and $\sigma(A)$ satisfies~\eqref{eq:subdiag}. By the induction hypothesis,
$\sigma(\Lambda)=\gen{\sigma(A),\sigma(D_1),\dots,\sigma(D_m)}$ is dense in
$\T_{n-1}$.

\emph{Step 1.} We show that
\begin{equation}\label{eq:block}
\begin{pmatrix}R&0\\0&x\end{pmatrix}\in\cl\Lambda\qquad
\text{for all }R\in\T_{n-1},\ x\in\K .
\end{equation}
Let $R\in\T_{n-1}$. There are $L_k\in\Lambda$ with $\sigma(L_k)\to R$; write
\[
L_k=\begin{pmatrix}R_k&0\\V_k&x_k\end{pmatrix},\qquad R_k\to R,\quad x_k\neq0,
\]
where $x_k\neq0$ because $L_k$ is invertible. Choose nonzero scalars
$\varepsilon_k$ with $\varepsilon_kV_k\to0$. Since all diagonal matrices lie in
$\cl\Lambda$,
\[
\begin{pmatrix}I_{n-1}&0\\0&\varepsilon_k\end{pmatrix}L_k
\begin{pmatrix}I_{n-1}&0\\0&x/(\varepsilon_kx_k)\end{pmatrix}
=\begin{pmatrix}R_k&0\\\varepsilon_kV_k&x\end{pmatrix}\in\cl\Lambda,
\]
and letting $k\to\infty$ gives~\eqref{eq:block}.

\emph{Step 2.} Write
\[
A=\begin{pmatrix}A_0&0\\V&a_n\end{pmatrix},\qquad V=(v_1,\dots,v_{n-1}),
\]
where $A_0=\sigma(A)$ is invertible and $v_{n-1}=A_{n,n-1}\neq0$
by~\eqref{eq:subdiag}. Let
\[
B=\begin{pmatrix}U&0\\W&z\end{pmatrix}\in\T_n,\qquad W=(w_1,\dots,w_{n-1}),
\quad w_{n-1}\neq0 .
\]
Define $U'\in\T_{n-1}$ to agree with the identity matrix except in its last
row, which is
\[
U'_{n-1,j}=\frac{w_j-v_j}{v_{n-1}}\ (1\le j<n-1),\qquad
U'_{n-1,n-1}=\frac{w_{n-1}}{v_{n-1}} .
\]
Then $U'$ is invertible, and $VU'=W$: indeed
$(VU')_j=v_j+v_{n-1}U'_{n-1,j}=w_j$ for $j<n-1$, and
$(VU')_{n-1}=v_{n-1}U'_{n-1,n-1}=w_{n-1}$. Put $R=U(U')^{-1}A_0^{-1}\in\T_{n-1}$.
Then
\[
\begin{pmatrix}R&0\\0&1\end{pmatrix}
\begin{pmatrix}A_0&0\\V&a_n\end{pmatrix}
\begin{pmatrix}U'&0\\0&z/a_n\end{pmatrix}
=\begin{pmatrix}RA_0U'&0\\VU'&z\end{pmatrix}
=\begin{pmatrix}U&0\\W&z\end{pmatrix}=B .
\]
The outer factors lie in $\cl\Lambda$ by~\eqref{eq:block}, and $A\in\Lambda$;
hence $B\in\cl\Lambda$. The matrices $B$ with $B_{n,n-1}\neq0$ are dense in
$\T_n$, so $\cl\Lambda=\T_n$.
\end{proof}

\begin{proof}[Proof of Theorem~\ref{thm:main}]
The lower bound in (a) is Proposition~\ref{prop:lower}(b). For (b), take
$c_j=\sqrt{p_j}$ and, in the complex case, $\beta_j=\sqrt{q_j}$. By
Besicovitch's theorem, $1,c_1,\dots,c_n$ (and $\beta_1,\dots,\beta_n$ in the
complex case) are linearly independent over $\Q$, and $c_1>\dots>c_n>0$. By
Proposition~\ref{prop:diag}, the matrices $D_0=\diag(A_{11},\dots,A_{nn})$
and $D_1,\dots,D_n$ generate a dense subsemigroup of $\D_n(\K)$, and $D_0$
satisfies the first alternative in~\eqref{eq:incr}. Since $A$ has diagonal
$D_0$ and $A_{i+1,i}=1$, Theorem~\ref{thm:char} shows that
$\gen{A,D_1,\dots,D_n}$ is dense in $\T_n(\K)$. This also proves the existence
statement in (a).
\end{proof}

\section{Corollaries, remarks and questions}\label{sec:cor}

\begin{corollary}\label{cor:group}
Let $\T_n^\ast(\K)$ be the group of invertible $n\times n$ lower triangular
matrices. Then $\T_n^\ast(\K)$ has a dense subsemigroup, and hence a dense
subgroup, generated by $n+1$ elements, and $n+1$ is the least number of
generators of a dense subgroup (or subsemigroup) of $\T_n^\ast(\K)$.
\end{corollary}

\begin{proof}
The semigroup of Theorem~\ref{thm:main}(b) consists of invertible matrices and
is dense in $\T_n$, hence in the open subset $\T_n^\ast$. The lower bound is
Proposition~\ref{prop:lower}(c); a dense subsemigroup of $\T^\ast_n$ generates
a dense subgroup with the same generators.
\end{proof}

\begin{corollary}\label{cor:orbit}
The $(n+1)$-generator semigroup $\Lambda$ of Theorem~\ref{thm:main}(b) has a
dense orbit in $\K^n$: the orbit $\{Xe_1:X\in\Lambda\}$ of the first standard
basis vector is dense. More generally, $\{Xu: X\in\Lambda\}$ is dense in $\K^n$
for every $u\in\K^n$ with $u_1\neq0$.
\end{corollary}

\begin{proof}
For $u_1\neq0$ the map $X\mapsto Xu$ from $\T_n$ to $\K^n$ is continuous and
surjective (choose the first column of $X$ freely and all other entries $0$, up
to the scaling $1/u_1$), so it maps the dense set $\Lambda$ onto a dense set.
\end{proof}

By the results of Abels and Manoussos~\cite{AM} quoted in the introduction, an
\emph{abelian} semigroup of complex triangular non-diagonalizable matrices
needs $n+2$ generators to have a dense orbit; Corollary~\ref{cor:orbit} shows
that dropping commutativity reduces this to $n+1$ even within the triangular
matrices. (Without the triangularity restriction, $2$ generators suffice
\cite{Ja2}.)

\begin{remark}\label{rem:nec}
The necessity argument in the proof of Theorem~\ref{thm:char} applies to any
generating set: if $\gen{X_1,\dots,X_k}$ is dense in $\T_n$, then for each
$1\le i\le n-1$ some generator $X_j$ has $(X_j)_{i+1,i}\neq0$. In
Theorem~\ref{thm:char} a single generator carries all the off-diagonal
information.
\end{remark}

\begin{remark}
The analogous results for upper triangular matrices follow by transposition,
which is an anti-automorphism of the matrix semigroup and maps semigroups
generated by $\mathcal{G}$ onto semigroups generated by the transposes.
\end{remark}

\begin{question}
Does Theorem~\ref{thm:char} remain true when the moduli $|a_1|,\dots,|a_n|$ are
pairwise distinct but not monotone, or even without any assumption on $D_0$
beyond invertibility? The proof of Lemma~\ref{lem:diagclosure} uses
monotonicity through the estimate of Lemma~\ref{lem:growth}, which controls
$(A^k)_{ij}$ by $|a_i|^k$ only when $|a_i|$ is the largest modulus among
$a_j,\dots,a_i$.
\end{question}

\begin{question}
Let $G$ be a connected solvable linear algebraic group over $\K$, and let
$\mathcal{S}$ be its Zariski closure in the space of matrices (for instance, a
semigroup of block lower triangular matrices). What is the minimal number of
generators of a dense subsemigroup of $\mathcal{S}$? Is it always equal to the
corresponding number for a maximal torus, as Theorem~\ref{thm:main} shows for
$\T_n$?
\end{question}

\end{document}